\documentclass[11pt]{article}
\usepackage[margin=1in]{geometry}
\usepackage[utf8]{inputenc}
\usepackage[english]{babel}
\usepackage{amsmath,amsfonts,amssymb,amsthm,color,mathrsfs,mathtools,enumitem}
\usepackage[hidelinks]{hyperref}

\theoremstyle{definition}
\newtheorem{theorem}{Theorem}[section]
\newtheorem{proposition}[theorem]{Proposition}
\newtheorem{lemma}[theorem]{Lemma}
\newtheorem{corollary}[theorem]{Corollary}
\theoremstyle{remark}
\newtheorem{remark}[theorem]{Remark}
\numberwithin{equation}{section}

\newcommand{\mbb}{\mathbb}
\newcommand{\mcal}{\mathcal}
\newcommand{\mscr}{\mathscr}
\newcommand{\msf}{\mathsf}
\newcommand{\im}{\mathrm{i}}
\newcommand{\Z}{\mbb{Z}}
\newcommand{\R}{\mbb{R}}
\newcommand{\C}{\mbb{C}}
\newcommand{\del}{\partial}
\newcommand{\grad}{\nabla}
\newcommand{\abs}[1]{\left\lvert #1\right\rvert}
\newcommand{\norm}[1]{\left\lVert #1\right\rVert}

\title{Linear and Nonlinear Instability in Two-Dimensional Inertial Magnetohydrodynamics}
\author{Hyungjun Choi}
\date{\today}

\begin{document}
\maketitle

\begin{abstract}
We study spectral and nonlinear instability for some steady states of a two-dimensional inertial magnetohydrodynamic system on the torus. For equilibria with aligned velocity and magnetic fields, the linearized problem reduces to a matrix recurrence of Fourier coefficients, and a stable matrix continued fraction produces a smooth eigenmode. For a purely magnetic family, the square of the linearized operator admits a compact self-adjoint reduction on each transverse Fourier block that classifies the entire unstable spectrum. In an invariant symmetry class, we identify the unstable branches and a spectral gap below the leading eigenspace. The spectral gap then permits an approximate-trajectory construction that proves nonlinear instability on the logarithmic time scale. Our study is motivated by the physics literature, where the nonlinear instability of purely magnetic tearing modes is invoked to explain magnetic reconnection.
\end{abstract}

\section{Introduction}
On the flat torus $\mbb T^2=(\R/2\pi\Z)^2$, we study the inertial magnetohydrodynamic system
\begin{equation} \label{eq:imh-system}
\begin{aligned} \del_t\omega+\grad^\perp\psi\cdot\grad\omega+\grad^\perp\Phi\cdot\grad F&=0,\\ \del_tF+\grad^\perp\psi\cdot\grad F&=0,\\ \Delta\psi=\omega,\qquad (1-\Delta)\Phi&=F, \end{aligned}
\end{equation}
where $\grad^\perp f=(-\del_{x_2}f,\del_{x_1}f)$, $\psi$ is the fluid stream function, $\omega$ is the vorticity, and $\Phi$ is the magnetic flux function. The generalized flux $F=(1- d_e^2 \Delta)\Phi$ incorporates electron inertia. Here, the electron skin depth $d_e$ is normalized to one. This system is a two-dimensional reduction of MHD with electron inertia in a thin slab. In the underlying slab description, $F$ is proportional to the electron canonical momentum in the ignorable direction. See \cite{LingamMorrisonTassi2015} and references therein for further discussion on derivation of this model.

Magnetic reconnection and relaxation motivate the study of inertial MHD. In resistive magnetohydrodynamics, reconnection is tied to magnetic diffusion, but laboratory and astrophysical plasmas can reconnect on time scales too short to be explained by collisional resistivity alone~\cite{Taylor1986}. In models with electron inertia, the transported flux $F$ differs from the magnetic flux $\Phi$~\cite{SchepPegoraroKuvshinov1994,AbdelhamidKawazuraYoshida2015}. Since the level sets of $\Phi$ describe planar magnetic field lines, these lines need not move with the fluid even though $F$ is transported.

Ottaviani and Porcelli studied growing perturbations associated with magnetic reconnection (tearing modes) in this model~\cite{OttavianiPorcelli1993}. They found accelerating growth and narrowing current layers during early nonlinear evolution, in a regime where the electron skin depth is small relative to the equilibrium scale. In other related models, conservation laws help explain thin current and vorticity layers~\cite{CafaroEtAl1998} and the eventual saturation of perturbation growth~\cite{GrassoEtAl2001}. Other work develops a Hamiltonian formulation~\cite{TassiEtAl2008} and studies rapid nonlinear growth with electron-temperature effects~\cite{HirotaHattoriMorrison2015}. These studies motivate our instability analysis.

This article proves spectral instability for a family of equilibria with aligned velocity and magnetic fields, and spectral and nonlinear instability for a purely magnetic sinusoidal family of \eqref{eq:imh-system}. Our nonlinear theorem establishes a fixed order-one escape from the equilibria in a Sobolev norm, not magnetic topology change or current-sheet formation; connecting this instability directly to reconnection remains an open problem.

Constantin and Hu~\cite{ConstantinHu2026} studied a related planar inertial model with the additional term $\grad^\perp\Phi\cdot\grad\omega$ on the left-hand side of the $F$ equation. They proved global regularity and constructed solutions with magnetic topology change without resistivity. Their model transports the two fields $F\pm\omega$ with stream functions $\psi\pm\Phi$. This particular transport structure is absent for \eqref{eq:imh-system}.

Liao, Lin, and Zhu~\cite{LiaoLinZhu2026} studied the stability and instability of Kelvin--Stuart magnetic islands, which are purely magnetic steady states of planar ideal MHD. For these magnetic islands, they prove spectral and nonlinear orbital stability for co-periodic perturbations. They also prove linear instability for perturbations with twice the equilibrium period. The latter gives a rigorous linear form of coalescence instability, associated with the tendency of neighboring islands to merge \cite{PriestForbes2000}.

\subsection{Equilibrium families and main results}

We analyze two families of one-dimensional steady states. The first carries aligned velocity and magnetic fields,
\[ \psi_\ast=\sin(kx_1),\qquad \Phi_\ast=a\sin(kx_1), \]
where $k\geq2$ is an integer and $a\in\R$. The second is the purely magnetic family
\[ \psi_\ast=0,\qquad \Phi_\ast=\sin(kx_1), \]
where $k\geq1$ is an integer. In either case, $\psi_\ast$ and $\Phi_\ast$ are Laplacian eigenfunctions with the common eigenvalue $-k^2$. Thus $\omega_\ast=\Delta\psi_\ast$ and $F_\ast=(1-\Delta)\Phi_\ast$ are also constant multiples of $\sin(kx_1)$, and all nonlinear Poisson-bracket terms vanish identically.

The main theorems may be summarized as follows.
\begin{enumerate}[label=\textup{(\roman*)},leftmargin=*]
\item Theorem \ref{thm:main-inertial-mhd-instability} proves that every aligned equilibrium with $k\geq2$ and $a^2(1+k^2)<2$ has a positive real eigenvalue and a smooth Fourier eigenfunction. 
\item Theorem \ref{thm:full-unstable-spectrum} and Proposition \ref{prop:sobolev-semigroup} describe the unstable and Fredholm essential spectra of the purely magnetic family: the unstable spectrum consists of finitely many positive eigenvalues obtained from negative eigenvalues of compact self-adjoint operators, and the Fredholm essential spectrum is $\im[-k\sqrt{1+k^2},k\sqrt{1+k^2}]$. Within the specified symmetry class, Proposition \ref{prop:symmetry-branches}, Proposition \ref{prop:continued-fraction}, and Theorem \ref{thm:spectral-gap} give exactly one simple positive branch for each $1\leq n<k$, a scalar continued-fraction formula with quantitative bounds, and a positive gap below the leading eigenspace; for $k=1$ there is no growing spectral value.
\item Theorem \ref{thm:nonlinear-instability} proves nonlinear instability of the purely magnetic equilibrium for every $k\geq2$. For each integer $s\geq4$, an initial perturbation of size $\epsilon$ along a normalized leading eigenfunction in the invariant symmetry class reaches size comparable to a fixed small $\delta$ in $H^{s-1}\times H^s$ at time $T_\epsilon=\Lambda_k^{-1}\log(\delta/\epsilon)$, where $\Lambda_k$ is the largest growth rate in that class and $\delta$ is independent of $\epsilon$.
\end{enumerate}

\subsection{Proof strategies}

The proof of the first result follows the continued-fraction method for Fourier coefficients originating in the work of Meshalkin and Sinai on Kolmogorov flow~\cite{MeshalkinSinai1961}. A normal-mode ansatz turns the linearized equations into a nearest-neighbor recurrence for Fourier coefficients. In the hydrodynamic case $a=0$ this recurrence separates into scalar chains; for $a\neq0$ it becomes genuinely matrix-valued. We construct the stable subspace that generates the exponentially decaying tail, identify that subspace through a convergent matrix continued fraction, and compare its small- and large-growth-rate limits. A sign change of the resulting matching determinant then produces a positive eigenvalue. For Euler linear instability theory, see~\cite{Lin2003,VishikFriedlander1993,FriedlanderVishik1991,CaoLaboraColomboDolceVentura2026}. 

For the purely magnetic equilibrium, the absence of background velocity removes the diagonal transport terms. After fixing a transverse Fourier frequency, we relate the nonzero spectrum of the squared linearized operator on that Fourier block to the spectrum of a compact self-adjoint operator. Its negative index is computed directly from the Fourier symbol, giving the complete list of unstable branches. Uniform coercive estimates on high transverse modes exclude additional right-half-plane spectrum from the infinite direct sum, while a pseudodifferential argument identifies the essential spectrum. Within the symmetry class, reflection parity selects one scalar Fourier chain for each $1\leq n<k$; a scalar continued fraction gives quantitative eigenvalue bounds, and the finiteness of these branches yields a positive gap below the leading eigenspace.

To pass from spectral to nonlinear instability, we use the high-order approximate-trajectory method introduced by Grenier~\cite{Grenier2000}, in the broader tradition of nonlinear instability results for ideal fluids~\cite{FriedlanderStraussVishik1997,Lin2004}. The derivative loss in the quadratic transport term prevents a direct fixed-regularity perturbation argument. We instead combine semigroup bounds with the same exponential growth rate at each Sobolev regularity, tame energy estimates, and a recursively constructed approximate solution whose residual is of arbitrarily high order. A bootstrap then controls the exact solution until the leading mode reaches a fixed size $\delta$ at time $T_\epsilon=\Lambda_k^{-1}\log(\delta/\epsilon)$.

\section{Instability of aligned velocity--magnetic equilibria}

The first result supplies an unstable eigenmode for an aligned family. The case $a=0$ is hydrodynamic, whereas the case $a\neq0$ is genuinely coupled.

\begin{theorem} \label{thm:main-inertial-mhd-instability}
Let $k\geq2$ be an integer and let $a\in\R$ satisfy
\begin{equation} \label{eq:parameter-condition}
a^2(1+k^2)<2.
\end{equation}
Then the steady state $\psi_\ast=\sin(kx_1)$, $\Phi_\ast=a\sin(kx_1)$ is linearly unstable. More precisely, its linearization has a nonzero smooth eigenfunction in the transverse Fourier mode $e^{\im x_2}$ with a positive real eigenvalue.
\end{theorem}

\subsection{Linearization and the Fourier recurrence}

Let $(\omega,F)$ denote a perturbation and recover $(\psi,\Phi)$ from the elliptic equations in \eqref{eq:imh-system}. Write $u_\ast=\grad^\perp\psi_\ast=(0,k\cos(kx_1))$. Retaining the linear terms and simplifying gives
\begin{equation} \label{eq:mixed-linearized-system}
\begin{aligned} \del_t\omega+u_\ast\cdot\grad\left[(1+k^2\Delta^{-1})\omega+a\{1-(1+k^2)(1-\Delta)^{-1}\}F\right]&=0,\\ \del_tF+u_\ast\cdot\grad\left[F-a(1+k^2)\Delta^{-1}\omega\right]&=0. \end{aligned}
\end{equation}
We seek a normal mode
\begin{equation} \label{eq:mixed-normal-mode}
\begin{pmatrix}\omega\\F\end{pmatrix}=e^{\lambda t}e^{\im x_2}\sum_{j=0}^\infty v_j\cos(jkx_1),\qquad v_j\in\C^2.
\end{equation}
Substitution of \eqref{eq:mixed-normal-mode} into \eqref{eq:mixed-linearized-system} gives
\[ \lambda v_0=-\frac{\im k}{2}M_1v_1,\qquad \lambda v_j=-\frac{\im k}{2}(M_{j-1}v_{j-1}+M_{j+1}v_{j+1}),\quad j\geq1, \]
where
\[ M_j=\begin{pmatrix}1-\dfrac{k^2}{d_j}&a\left(1-\dfrac{1+k^2}{1+d_j}\right)\\[1.1em]\dfrac{a(1+k^2)}{d_j}&1\end{pmatrix},\qquad d_j=j^2k^2+1,\quad j\geq1. \]
The constant cosine mode has the exceptional matrix
\[ M_0=\begin{pmatrix}2(1-k^2)&a(1-k^2)\\2a(1+k^2)&2\end{pmatrix}. \]
Thus, with $\mu=2\lambda/k$ and $b_j=\im^{j-1}v_j$, the spectral problem becomes the real recurrence
\begin{equation} \label{eq:mixed-recurrence}
\mu b_0=-M_1b_1,\qquad \mu b_j=M_{j-1}b_{j-1}-M_{j+1}b_{j+1},\quad j\geq1.
\end{equation}
We will construct a real exponentially decaying solution of \eqref{eq:mixed-recurrence}.

The matrices satisfy, for $j\geq1$,
\begin{equation} \label{eq:matrix-determinants}
\det M_j=\left(1-\frac{k^2}{d_j}\right)\left(1-\frac{a^2(1+k^2)}{1+d_j}\right)>0,\qquad \det M_0=2(1-k^2)\{2-a^2(1+k^2)\}<0.
\end{equation}
Both factors in $\det M_j$ are nondecreasing in $j$, so $1>\det M_j\geq\det M_1>0$. The entries of $M_j$ and $M_j^{-1}$ are uniformly bounded, and
\begin{equation} \label{eq:matrix-limit}
M_\infty=\begin{pmatrix}1&a\\0&1\end{pmatrix},\qquad \norm{M_j-M_\infty}+\lVert M_j^{-1}-M_\infty^{-1} \rVert \leq\frac{C}{j^2},\quad j\geq1.
\end{equation}
Unless otherwise stated, the matrix norm is the operator norm induced by the Euclidean norm.

\begin{lemma} \label{lem:positive-symmetric-part}
Under \eqref{eq:parameter-condition}, the symmetric part of $M_j^{-1}$ is positive definite, uniformly in $j\geq1$.
\end{lemma}
\begin{proof}
Write $M_j=\begin{pmatrix}\alpha_j&\beta_j\\\gamma_j&1\end{pmatrix}$, where $\alpha_j=1-k^2/d_j$, $\beta_j=a(d_j-k^2)/(d_j+1)$, and $\gamma_j=a(1+k^2)/d_j$. The symmetric part of $M_j^{-1}$ is
\[ (\det M_j)^{-1}\begin{pmatrix}1&-(\beta_j+\gamma_j)/2\\-(\beta_j+\gamma_j)/2&\alpha_j\end{pmatrix}. \]
It is positive definite precisely when $4\alpha_j>(\beta_j+\gamma_j)^2$. A simplification gives
\[ \beta_j+\gamma_j=a\left(1+\frac{1+k^2}{d_j(d_j+1)}\right). \]
As $j$ increases, $\abs{\beta_j+\gamma_j}$ is nonincreasing and $\alpha_j$ increases, so it is enough to check $j=1$. By \eqref{eq:parameter-condition},
\[ (\beta_1+\gamma_1)^2 = a^2 \left(\frac{k^2+3}{k^2+2}\right)^2 <\frac{2}{1+k^2}\left(\frac{k^2+3}{k^2+2}\right)^2<\frac{4}{1+k^2}=4\alpha_1, \]
where the second strict inequality holds for $k\geq2$. Since $\det M_j\leq1$ and $4\alpha_j-(\beta_j+\gamma_j)^2\geq4\alpha_1-(\beta_1+\gamma_1)^2>0$, the symmetric part of $M_j^{-1}$ is uniformly positive definite in $j$.
\end{proof}

\subsection{The stable matrix continued fraction}

\begin{lemma} \label{lem:stable-graph}
For every $\mu>0$, the initial pairs $(b_0,b_1)$ generating exponentially decaying solutions of the tail recurrence in \eqref{eq:mixed-recurrence} form a two-dimensional subspace. This subspace is the graph of a unique real matrix $R(\mu)$,
\begin{equation} \label{eq:stable-graph}
b_1=R(\mu)b_0,
\end{equation}
and $R:(0,\infty)\to\R^{2\times2}$ is continuous.
\end{lemma}

\begin{proof}
We first examine the constant recurrence obtained by replacing $M_j$ with $M_\infty$. Its exponentially decaying solutions are $b_j=R_\infty^jc$, where
\begin{equation} \label{eq:constant-tail-ratio}
R_\infty=\begin{pmatrix}\rho&\sigma\\0&\rho\end{pmatrix},\qquad \rho=\frac{\sqrt{\mu^2+4}-\mu}{2}\in(0,1),\qquad \sigma=\frac{a(1-\rho^2)}{\mu+2\rho}.
\end{equation}
Indeed, $\mu R_\infty+M_\infty R_\infty^2=M_\infty$. For every $\theta\in(\rho,1)$, one has $\lVert R_\infty^j \rVert \leq C_\theta\theta^j$, so the limiting recurrence relation has a two-dimensional stable space.

Write the recurrence \eqref{eq:mixed-recurrence} as $\mathbf{b}_{j+1}=\mcal A_j(\mu)\mathbf{b}_j$ for $\mathbf{b}_j=(b_j,b_{j-1})^{\mathsf T}$, where
\[ \mcal A_j(\mu)=\begin{pmatrix}-\mu M_{j+1}^{-1}&M_{j+1}^{-1}M_{j-1}\\I&0\end{pmatrix}. \]
The limiting matrix $\mcal A_\infty(\mu) = \begin{pmatrix}-\mu M_\infty^{-1} & I\\I&0\end{pmatrix}$ is hyperbolic: its characteristic roots are $\rho$ and $-\rho^{-1}$, each with algebraic multiplicity two, and its stable generalized eigenspace is the graph $b_j=R_\infty b_{j-1}$. Since $\sum_{j\geq1}\norm{\mcal A_j-\mcal A_\infty}<\infty$ by \eqref{eq:matrix-limit}, roughness of discrete exponential dichotomies~\cite[Chapter~3]{Coppel1978} gives, for all sufficiently large $J$, an exact two-dimensional stable space at level $J$; it consists of all solutions that decay exponentially as $j\to\infty$.

We now identify this space by a convergent matrix continued fraction. For a matrix $R$, set $T_j(R)=(\mu I+M_{j+1}R)^{-1}M_{j-1}$. The identities $b_j=R_jb_{j-1}$ satisfy the recurrence \eqref{eq:mixed-recurrence} exactly when $R_j=T_j(R_{j+1})$. Use the vector norm $\abs{(x,y)}_\delta=\max\{\abs{x},\delta\abs{y}\}$ with $\delta=\max\{1, \frac{2\abs{\sigma}}{1-\rho}\}$. Its induced matrix norm satisfies $\norm{R_\infty}_\delta\leq \frac{1+\rho}{2}<1$. Choose $\theta\in(\frac{1+\rho}{2},1)$. By \eqref{eq:matrix-limit}, there are $\epsilon>0$ and $J$ such that, whenever $\norm{R-R_\infty}_\delta\leq \epsilon$ and $j\geq J$,
\[ \begin{aligned} \norm{(\mu I+M_{j+1}R)^{-1}M_{j+1}}_\delta&\leq\theta,\qquad &\norm{(\mu I+M_{j+1}R)^{-1}M_{j-1}}_\delta&\leq\theta,\\ \norm{T_j(R_\infty)-R_\infty}_\delta&\leq(1-\theta^2)\epsilon. \end{aligned} \]
The resolvent identity gives
\begin{equation} \label{eq:tail-exact-difference}
T_j(X)-T_j(Y)=(\mu I+M_{j+1}X)^{-1}M_{j+1}(Y-X)(\mu I+M_{j+1}Y)^{-1}M_{j-1}.
\end{equation}
Hence $T_j$ maps the closed $\epsilon$-ball about $R_\infty$ into itself and contracts there by the factor $\theta^2$. For $N\geq J$, start with $R_{N+1}^{(N)}=R_\infty$ and define $R_j^{(N)}=T_j(R_{j+1}^{(N)})$ backward. If $N'>N$, repeated contraction gives
\[ \lVert R_j^{(N')}-R_j^{(N)} \rVert_\delta\leq \epsilon\theta^{2(N-j+1)},\qquad J\leq j\leq N. \]
Thus $R_j^{(N)}$ converges for each fixed $j\geq J$, its limit satisfies the backward equation, and $R_j\to R_\infty$ as $j\to\infty$. These ratios generate a two-dimensional family of exponentially decaying solutions and therefore give the stable subspace. Propagating the stable subspace backward by the linear recurrence is legitimate through $j=1$ because $M_0,M_1,\ldots$ are invertible by \eqref{eq:matrix-determinants}.

It remains to show that this stable plane is a graph over $b_0$. Put $p_j=M_jb_j$. Then
\begin{equation} \label{eq:p-recurrence}
p_{j+1}=p_{j-1}-\mu M_j^{-1}p_j,\qquad j\geq1.
\end{equation}
If an exponentially decaying solution has $b_0=0$, multiply \eqref{eq:p-recurrence} by $p_j^{\mathsf T}$ and sum from $j=1$ to $N$. The nearest-neighbor terms cancel and give $\mu\sum_{j=1}^Np_j^{\mathsf T}M_j^{-1}p_j=p_1^{\mathsf T}p_0-p_N^{\mathsf T}p_{N+1}$. Passage to the limit gives $\sum_{j\geq1}p_j^{\mathsf T}M_j^{-1}p_j=0$, and Lemma \ref{lem:positive-symmetric-part} forces every $p_j$ to vanish. The projection of the constructed two-dimensional stable plane onto $b_0\in\R^2$ is therefore injective and hence bijective, proving the graph representation. This identity also proves that the constructed plane contains every decaying solution: choose a vector in the constructed plane with the same $b_0$ component as any such solution and apply the same uniqueness argument to their difference.

Finally, on each compact interval of $\mu>0$ the exponential dichotomy, tail contraction, and finite backward propagation are uniform. A continuous basis of the stable plane at level $J$ therefore propagates to a continuous basis at level zero. Its projection onto the $b_0$ coordinate is an invertible $2\times2$ matrix by the preceding transversality argument, so inversion of that projection shows that $R(\mu)$ is continuous.
\end{proof}

\subsection{Endpoint asymptotics and matching}

\begin{lemma} \label{lem:endpoints}
There is a constant $C\geq1$ such that
\begin{equation} \label{eq:endpoint-asymptotics}
\norm{M_1R(\mu)-M_0}\leq C\mu\quad(0<\mu\leq C^{-1}),\qquad \norm{R(\mu)-\mu^{-1}M_0}\leq C\mu^{-3}\quad(\mu\geq C).
\end{equation}
\end{lemma}

\begin{proof}
For the small-$\mu$ estimate, use the Euclidean operator norm and put
\[ \msf K=\sup_{j\geq1}\lVert M_j^{-1}\rVert,\qquad \kappa=\inf_{\substack{j\geq1\\\abs{z}=1}}z^{\mathsf T}M_j^{-1}z>0. \]
The positivity of $\kappa$ follows from Lemma \ref{lem:positive-symmetric-part}. Introduce $p_j=M_jb_j$, $\xi_j=(p_j+p_{j-1})/2$, and $\eta_j=(p_j-p_{j-1})/2$. Equation \eqref{eq:p-recurrence} becomes
\[ \xi_{j+1}=\xi_j-\frac{\mu}{2}M_j^{-1}(\xi_j+\eta_j),\qquad \eta_{j+1}=-\eta_j-\frac{\mu}{2}M_j^{-1}(\xi_j+\eta_j). \]
If the stable planes at levels $j+1$ and $j$ are respectively $\eta_{j+1}=X\xi_{j+1}$ and $\eta_j=Y\xi_j$, their exact relation is
\begin{equation} \label{eq:small-exact-graph}
\left(I+\frac{\mu}{2}M_j^{-1}-\frac{\mu}{2}XM_j^{-1}\right)Y=-X-\frac{\mu}{2}M_j^{-1}+\frac{\mu}{2}XM_j^{-1}.
\end{equation}
Denote the resulting map by $Y=\Gamma_j(X)$. Wherever the inverse exists, \eqref{eq:small-exact-graph} gives
\[ I+\Gamma_j(X)=\left[I+\frac{\mu}{2}(I-X)M_j^{-1}\right]^{-1}(I-X). \]
Differentiation and the resolvent identity therefore give
\begin{equation} \label{eq:small-exact-derivative}
D\Gamma_j(X)[Z]=-\left[I+\frac{\mu}{2}(I-X)M_j^{-1}\right]^{-1}Z\left[I+\frac{\mu}{2}M_j^{-1}(I-X)\right]^{-1}.
\end{equation}
Choose $0<\varepsilon<1$ so that $2\msf K\varepsilon\leq\kappa$. If $\norm{X}\leq\varepsilon$, then the symmetric parts of both $(I-X)M_j^{-1}$ and $M_j^{-1}(I-X)$ are bounded below by $\frac{\kappa}{2}I$. For a real matrix $B$ whose symmetric part is at least $dI$, one has
\[ \abs{(I+\frac{\mu}{2}B)z}^2\geq(1+\mu d)\abs z^2. \]
Thus both inverses in \eqref{eq:small-exact-derivative} exist throughout this ball, and
\begin{equation} \label{eq:small-contraction}
\norm{D\Gamma_j(X)}\leq\left(1+\frac{\kappa\mu}{2}\right)^{-1}\leq1-c\mu,\qquad \norm{X}\leq\varepsilon,
\end{equation}
for sufficiently small $\mu>0$, with $c>0$ independent of $j$. This estimate requires no bound on $\Gamma_j(X)$. The limiting stable subspace for $(\xi_j,\eta_j)$ is the graph of $S_\infty=(R_\infty-I)(R_\infty+I)^{-1}$. From \eqref{eq:constant-tail-ratio}, we get $\norm{S_\infty}\leq C\mu$. Since $S_\infty$ solves the limiting equation,
\begin{equation} \label{eq:small-coefficient-identity}
\Gamma_j(S_\infty)-S_\infty=-\frac{\mu}{2}\left\{I+\frac{\mu}{2}(I-S_\infty)M_j^{-1}\right\}^{-1}(I-S_\infty)(M_j^{-1}-M_\infty^{-1})(I+S_\infty).
\end{equation}
Hence $\norm{\Gamma_j(S_\infty)-S_\infty}\leq C\mu j^{-2}$.

Choose $\mu>0$ sufficiently small that
\[ \norm{S_\infty}+\sum_{\ell=1}^{\infty}\norm{\Gamma_\ell(S_\infty)-S_\infty}\leq\frac{\varepsilon}{2}. \]
The backward iteration from infinity is made precise by finite cutoffs. For $N\geq1$, define $S_{N+1}^{(N)}=S_\infty$ and $S_j^{(N)}=\Gamma_j(S_{j+1}^{(N)})$ for $j=N,\ldots,1$. By backward induction, \eqref{eq:small-contraction} applies on the segment joining $S_\infty$ to each preceding iterate and gives
\begin{equation} \label{eq:small-cutoff-bound}
\begin{aligned}
\norm{S_j^{(N)}-S_\infty}&\leq\sum_{\ell=j}^N(1-c\mu)^{\ell-j}\norm{\Gamma_\ell(S_\infty)-S_\infty}\\
&\leq C\mu\sum_{\ell=j}^{\infty}\ell^{-2},\qquad 1\leq j\leq N.
\end{aligned}
\end{equation}
Indeed, the first bound and the choice of $\mu$ imply $\norm{S_j^{(N)}}\leq\varepsilon/2$, so all iterates and the required segments remain inside the ball on which \eqref{eq:small-contraction} holds. If $N'>N$, the terminal discrepancy at level $N+1$ is bounded uniformly, and repeated contraction gives $\norm{S_j^{(N')}-S_j^{(N)}}\leq C(1-c\mu)^{N-j+1}$. Thus $S_j^{(N)}$ converges for each fixed $j$ to a graph $S_j$ satisfying $\norm{S_j-S_\infty}\leq C\mu\sum_{\ell=j}^{\infty}\ell^{-2}$. In particular, this indexed tail bound proves $S_j\to S_\infty$ as $j\to\infty$.

This limiting graph corresponds to the stable subspace from Lemma \ref{lem:stable-graph} after returning to the original variables $b_j$ and $R_j$. Indeed, define
\[ R_j=M_j^{-1}(I+S_j)(I-S_j)^{-1}M_{j-1}. \]
Then $b_j=R_jb_{j-1}$ solves the recurrence. Since $S_j\to S_\infty$ and $M_\infty$ commutes with $R_\infty$, we have $R_j\to M_\infty^{-1}R_\infty M_\infty=R_\infty$. Hence these solutions decay exponentially, and uniqueness in Lemma \ref{lem:stable-graph} identifies their level-one graph with $R(\mu)$. In particular, \eqref{eq:small-cutoff-bound} gives $\norm{S_1}\leq C\mu$. On the stable subspace, $p_0=(I-S_1)\xi_1$ and $p_1=(I+S_1)\xi_1$, so
\[ M_1R(\mu)=(I+S_1)(I-S_1)^{-1}M_0. \]
This proves $\norm{M_1R(\mu)-M_0}\leq C\mu$.

For the large-$\mu$ estimate, let $\msf M=\sup_{j\geq0}\norm{M_j}$. The backward maps $T_j$ preserve the ball $\norm{R}\leq2\msf M/\mu$. By \eqref{eq:tail-exact-difference}, they are contractions with factor $4\msf M^2/\mu^2$ when $\mu$ is large. Repeating the construction in the proof of Lemma \ref{lem:stable-graph} produces a sequence $R_j$ in the ball $\norm{R}\leq2\msf M/\mu$. Since $R(\mu)=T_1(R_2)$, the resolvent identity gives
\[ R(\mu)-\mu^{-1}M_0=-\mu^{-1}(\mu I+M_2R_2)^{-1}M_2R_2M_0, \]
which is $O(\mu^{-3})$.
\end{proof}

\begin{proof}[Proof of Theorem \ref{thm:main-inertial-mhd-instability}]
By Lemma \ref{lem:stable-graph}, a decaying tail satisfies the zeroth-mode equation in \eqref{eq:mixed-recurrence} if and only if
\[ \{\mu I+M_1R(\mu)\}b_0=0. \]
Define $f(\mu)=\det\{\mu I+M_1R(\mu)\}$. This is a continuous real-valued function. The small-$\mu$ estimate in Lemma \ref{lem:endpoints}, together with $\det M_0<0$, gives $f(\mu)<0$ for sufficiently small $\mu>0$. The large-$\mu$ estimate gives
\[ \mu I+M_1R(\mu)=\mu\left[I+\mu^{-2}M_1M_0+O(\mu^{-4})\right], \]
so $f(\mu)>0$ for sufficiently large $\mu$. The intermediate value theorem supplies $\mu_\ast>0$ with $f(\mu_\ast)=0$.

Choose $0\neq b_0\in\ker\{\mu_\ast I+M_1R(\mu_\ast)\}$ and let $(b_j)_{j\geq0}$ be the associated exponentially decaying sequence. Returning to the original coefficients, we obtain the nonzero normal mode
\[ e^{\lambda_\ast t}e^{\im x_2}\sum_{j=0}^\infty \im^{1-j}b_j\cos(jkx_1),\qquad \lambda_\ast=\frac{k\mu_\ast}{2}>0. \]
Exponential decay of Fourier coefficients makes the eigenfunction smooth.
\end{proof}

\section{Linear analysis around the purely magnetic equilibrium}

\subsection{Equilibrium, phase spaces, and linearization}

For a positive integer $k$, consider the purely magnetic steady state of \eqref{eq:imh-system}
\begin{equation} \label{eq:pure-equilibrium}
\psi_\ast=0,\qquad \Phi_\ast=\sin(kx_1),\qquad \omega_\ast=0,\qquad F_\ast=(1+k^2)\sin(kx_1).
\end{equation}
Let $(\omega,F)$ denote a perturbation and recover $(\psi,\Phi)$ from the elliptic equations in \eqref{eq:imh-system}. Since
\[ \grad^\perp\Phi_\ast=(0,k\cos(kx_1)),\qquad \grad F_\ast=(k(1+k^2)\cos(kx_1),0), \]
the linearized equations are
\begin{equation} \label{eq:linearized-system}
\begin{aligned}
\del_t\omega+k\cos(kx_1)\del_{x_2}A_kF&=0,\\
\del_tF+k\cos(kx_1)\del_{x_2}B_k\omega&=0,
\end{aligned}
\end{equation}
where
\begin{equation} \label{eq:Ak-Bk}
A_k=I-(1+k^2)(1-\Delta)^{-1},\qquad B_k=-(1+k^2)\Delta^{-1}.
\end{equation}
Thus $\del_t(\omega,F)=\mscr L_k(\omega,F)$ with
\[ \mscr L_k=-k\cos(kx_1)\del_{x_2}\begin{pmatrix}0&A_k\\B_k&0\end{pmatrix}. \]
We use the complexified phase spaces
\[ X^s=H^{s-1}_0(\mbb T^2;\C)\times H^s(\mbb T^2;\C) \ni (\omega, F),\qquad s\geq0, \]
where the subscript zero denotes mean-zero vorticity. The multiplier $A_k$ has order zero, $B_k$ has order $-2$, and $\del_{x_2}$ loses one derivative. It follows that $\mscr L_k$ is bounded on every $X^s$ and therefore generates a uniformly continuous group there.

\subsection{Transverse Fourier blocks and the self-adjoint reduction}

The coefficients are independent of $x_2$, so fix $n\neq0$, suppress the factor $e^{\im nx_2}$, and write $X_n^s=H^{s-1}(\mbb T_{x_1};\C)\times H^s(\mbb T_{x_1};\C)$ for the $n$-th Fourier mode subspace in $x_2$, equipped with the norm inherited from $X^s$. The restricted multipliers have symbols
\begin{equation} \label{eq:full-symbols}
A_{k,n}e^{\im mx_1}=\frac{m^2+n^2-k^2}{m^2+n^2+1}e^{\im mx_1},\qquad B_{k,n}e^{\im mx_1}=\frac{1+k^2}{m^2+n^2}e^{\im mx_1}.
\end{equation}
Put $C_k=\cos(kx_1)$, $P=C_kA_{k,n}$, and $Q=C_kB_{k,n}$. Then
\begin{equation} \label{eq:squared-mode-operator}
\mscr L_{k,n}=-\im nk\begin{pmatrix}0&P\\Q&0\end{pmatrix},\qquad \mscr L_{k,n}^2=-n^2k^2\begin{pmatrix}PQ&0\\0&QP\end{pmatrix}.
\end{equation}
Define
\begin{equation} \label{eq:full-Kn}
D_{k,n}=C_kB_{k,n}C_k,\qquad K_{k,n}=D_{k,n}^{1/2}A_{k,n}D_{k,n}^{1/2}.
\end{equation}
Since the symbol of $B_{k,n}$ is positive and tends to zero as $|m|\to\infty$, the operator $B_{k,n}$ is self-adjoint, positive, and compact on $L^2(\mbb T_{x_1})$. The operators $B_{k,n}$ and $C_k$ are injective on $L^2(\mbb T_{x_1})$. Thus $D_{k,n}$ is self-adjoint, positive, compact, and injective. Since $A_{k,n}$ is bounded and self-adjoint, $K_{k,n}$ is compact and self-adjoint. Therefore, every nonzero spectral point of $K_{k,n}$ is an isolated real eigenvalue of finite multiplicity, with only zero as a possible accumulation point. Proposition \ref{prop:one-mode-spectrum} transfers this discrete spectrum to $\mscr L_{k,n}$ on each Fourier block.

\begin{proposition} \label{prop:one-mode-spectrum}
For each fixed $n\neq0$, the operator $\mscr L_{k,n}$ is compact on $X_n^s$ for any $s\geq 0$. The nonzero spectral values of $\mscr L_{k,n}$ are precisely the roots of
\[ z^2=-n^2k^2\kappa,\qquad \kappa\in\sigma(K_{k,n})\setminus\{0\}. \]
Moreover, the algebraic multiplicity of $z$ for $\mscr L_{k,n}$ is identical to that of $\kappa$ for $K_{k,n}$. In addition, an eigenfunction of $\mscr L_{k,n}$ with nonzero eigenvalue is smooth.
\end{proposition}

\begin{proof}
Compactness on $X_n^s$ follows directly from the derivative balance: $P:H^s\to H^s\hookrightarrow H^{s-1}$ and $Q:H^{s-1}\to H^{s+1}\hookrightarrow H^s$, where the displayed embeddings are compact. We first prove the spectral correspondence for $s=0$, so that $PQ$ acts on $H^{-1}$ and $QP=D_{k,n}A_{k,n}$ acts on $L^2$. For bounded maps $S:Y\to X$ and $T:X\to Y$, the identity
\[ (z-TS)^{-1}=z^{-1}\bigl[I+T(z-ST)^{-1}S\bigr],\qquad z\neq0, \]
and its counterpart with $S,T$ interchanged show that $\sigma(ST)\setminus\{0\}=\sigma(TS)\setminus\{0\}$. The intertwining identities $T(ST-z)^m=(TS-z)^mT$ and $S(TS-z)^m=(ST-z)^mS$, or equivalently the corresponding Riesz projections, show that Jordan block dimensions agree for any $z\neq0$. Apply this first to $P:L^2\to H^{-1}$ and $Q:H^{-1}\to L^2$, and then on $L^2$ with $S=D_{k,n}^{1/2}$ and $T=D_{k,n}^{1/2}A_{k,n}$. This gives
\[ \sigma(PQ)\setminus\{0\}=\sigma(QP)\setminus\{0\}=\sigma(K_{k,n})\setminus\{0\}. \]
Spectral mapping in \eqref{eq:squared-mode-operator} gives the asserted relation between $z$ and $\kappa$.

If $\kappa$ has algebraic multiplicity $m$ for $K_{k,n}$, then it has multiplicity $m$ for each diagonal product in \eqref{eq:squared-mode-operator}; hence $-n^2k^2\kappa$ has multiplicity $2m$ for $\mscr L_{k,n}^2$. The involution $J=\operatorname{diag}(I,-I)$ satisfies $J\mscr L_{k,n}J=-\mscr L_{k,n}$, so the generalized eigenspaces at the two roots $z$ and $-z$ have equal dimension. Their direct sum is the generalized eigenspace of $\mscr L_{k,n}^2$ for $z^2$, and each root therefore has algebraic multiplicity $m$.

It remains to justify that this description is independent of $s$. We claim that every generalized eigenvector of $\mscr L_{k,n}$ at $z\neq0$ is smooth. Indeed, $\mscr L_{k,n}^2$ gains two derivatives on both components, while its restriction to the finite-dimensional generalized eigenspace at $z$ is invertible; the inverse on that space is a polynomial in $\mscr L_{k,n}^2$. Thus $U=(\mscr L_{k,n}^2|_E)^{-1}\mscr L_{k,n}^2U$ gains two derivatives, and iteration gives $U\in C^\infty$. Consequently the generalized eigenspaces, their dimensions, and all nonzero spectral values are the same on every $X_n^s$.
\end{proof}

\begin{lemma} \label{lem:negative-index}
The number of negative eigenvalues of $K_{k,n}$, counted with multiplicity, is
\[ N_{k,n}=\#\{m\in\Z:m^2+n^2<k^2\}. \]
\end{lemma}

\begin{proof}
For every $g\in L^2(\mbb T_{x_1})$,
\[ \langle K_{k,n}g,g\rangle=\langle A_{k,n}D_{k,n}^{1/2}g,D_{k,n}^{1/2}g\rangle. \]
The negative subspace $E_-$ of $A_{k,n}$ is spanned by the modes with $m^2+n^2<k^2$ and has dimension $N_{k,n}$. If a subspace $G$ is strictly negative for $K_{k,n}$, then $D_{k,n}^{1/2}G$ is strictly negative for $A_{k,n}$ and has the same dimension because $D_{k,n}^{1/2}$ is injective. Hence the negative index of $K_{k,n}$ is at most $N_{k,n}$. Conversely, since $E_-$ is finite-dimensional, there is a constant $c>0$ such that $\langle A_{k,n}f,f\rangle\leq-c\norm f^2$ on $E_-$. The range of $D_{k,n}^{1/2}$ is dense because the operator is self-adjoint and injective. Approximate a basis of $E_-$ sufficiently closely by vectors in this range; continuity of the quadratic form preserves strict negativity and linear independence. Pulling those approximants back through $D_{k,n}^{1/2}$ gives an $N_{k,n}$-dimensional negative subspace for $K_{k,n}$. For a compact self-adjoint operator, this negative index coincides with the number of negative eigenvalues counted with multiplicity.
\end{proof}

Enumerate these eigenvalues as $\{-\nu_{k,n,j}\}_{j=1}^{N_{k,n}}$. Proposition \ref{prop:one-mode-spectrum} gives the positive eigenvalues of $\mscr L_{k,n}$:
\begin{equation} \label{eq:unstable-eigenvalues}
\lambda_{k,n,j}=|n|k\sqrt{\nu_{k,n,j}},\qquad 1\leq j\leq N_{k,n}.
\end{equation}

Multiplication by $C_k$ changes an $x_1$-frequency by $\pm k$. Consequently,
\[ L^2(\mbb T_{x_1})=\bigoplus_{r=0}^{k-1}\mcal H_{k,r},\qquad \mcal H_{k,r}=\overline{\operatorname{span}}\{e^{\im(r+jk)x_1}:j\in\Z\}, \]
and every $\mcal H_{k,r}$ is invariant under $A_{k,n}$, $B_{k,n}$, $D_{k,n}$, and $K_{k,n}$. The number of positive eigenvalues in the $(n,r)$ block is therefore
\[ N_{k,n,r}=\#\{j\in\Z:(r+jk)^2+n^2<k^2\}. \]

\begin{theorem} \label{thm:full-unstable-spectrum}
The spectrum of $\mscr L_k$ in the open right half-plane is a finite set of isolated positive real eigenvalues of finite algebraic multiplicity. More precisely,
\[ \sigma(\mscr L_k)\cap\{\operatorname{Re}z>0\}=\bigcup_{0 < |n| < k}\{\lambda_{k,n,j}:1\leq j\leq N_{k,n}\}. \]
The union involves only the finitely many transverse frequencies $0<|n|<k$. Thus the unstable spectrum of the full operator is discrete and no additional unstable spectrum is created by the direct sum over high transverse modes.
\end{theorem}

\begin{proof}
The $n=0$ block vanishes. If $|n|\geq k$, then $A_{k,n} \succeq 0$, hence $K_{k,n}\succeq 0$ and Proposition \ref{prop:one-mode-spectrum} excludes unstable spectrum. For $0<|n|<k$, Proposition \ref{prop:one-mode-spectrum} and Lemma \ref{lem:negative-index} give exactly the finite list in the theorem.

It remains to rule out additional spectrum arising from the infinite direct sum over the transverse frequency $n$. For $|n|\geq k+1$, set
\[ E_n(\omega,F)=\langle B_{k,n}\omega,\omega\rangle_{H^1 \times H^{-1}} +\langle A_{k,n}F,F\rangle_{L^2 \times L^2}. \]
Write $\omega=\sum_{m\in\Z}\omega_me^{\im mx_1}$ and $F=\sum_{m\in\Z}F_me^{\im mx_1}$. Then
\[ E_n(\omega,F)=\sum_{m\in\Z}\left\{\frac{1+k^2}{m^2+n^2}|\omega_m|^2+\frac{m^2+n^2-k^2}{m^2+n^2+1}|F_m|^2\right\}. \]
Recall that
\[ \norm{(\omega,F)}_{X_n^0}^2=\sum_{m\in\Z}\left\{\frac{|\omega_m|^2}{1+m^2+n^2}+|F_m|^2\right\}. \]
For $r=m^2+n^2\geq(k+1)^2$, one has
\[ 1+k^2\leq(1+k^2)\frac{1+r}{r}\leq(1+k^2)\left(1+\frac{1}{(k+1)^2}\right),\qquad \frac{2k+1}{k^2+2k+2}\leq\frac{r-k^2}{r+1}<1. \]
Consequently there are constants $c,C>0$, depending only on $k$ and independent of $n$, such that
\[ c\norm{(\omega,F)}_{X_n^0}^2\leq E_n(\omega,F)\leq C\norm{(\omega,F)}_{X_n^0}^2,\qquad |n|\geq k+1. \]
Let $\langle\cdot,\cdot\rangle_{E_n}$ denote the inner product obtained by polarizing $E_n$. The bounded operator $\mscr L_{k,n}$ is skew-adjoint in this equivalent Hilbert norm, so $z-\mscr L_{k,n}$ is invertible whenever $\operatorname{Re}z\neq0$. Given $V\in X_n^0$, set $U=(z-\mscr L_{k,n})^{-1}V$. Taking the $E_n$ inner product of $(z-\mscr L_{k,n})U=V$ with $U$ and then taking real parts gives
\[ |\operatorname{Re}z|E_n(U)=\left|\operatorname{Re}\langle V,U\rangle_{E_n}\right|\leq E_n(V)^{1/2}E_n(U)^{1/2}. \]
Hence, uniformly for $|n|\geq k+1$,
\[ \norm{(z-\mscr L_{k,n})^{-1}}_{X_n^0\to X_n^0}\leq\frac{C}{|\operatorname{Re}z|},\qquad |n|\geq k+1. \]
For an orthogonal direct sum, $z\in\rho(\bigoplus_n\mscr L_{k,n})$ if and only if $z\in\rho(\mscr L_{k,n})$ for every $n$ and $\sup_n\norm{(z-\mscr L_{k,n})^{-1}}<\infty$. The preceding estimate is uniform on all high blocks $|n|\geq k+1$, and only finitely many blocks satisfy $|n|\leq k$. The two boundary blocks $|n|=k$ are not covered by the coercive energy because the symbol of $A_{k,n}$ vanishes at $m=0$; nevertheless, they are compact, and $A_{k,n}\succeq0$ implies $K_{k,n}\succeq0$, so Proposition \ref{prop:one-mode-spectrum} places their nonzero spectrum on the imaginary axis. Their resolvent norms are therefore finite for each fixed $z$ with $\operatorname{Re}z>0$. The same is immediate for the zero block, while the remaining finitely many blocks $0<|n|<k$ have precisely the positive eigenvalues listed in \eqref{eq:unstable-eigenvalues}. Thus every spectral point with positive real part comes from one of these finitely many compact blocks, and every such point is an isolated eigenvalue of finite algebraic multiplicity for the full operator.
\end{proof}

\begin{remark}
Write $X^0=\bigoplus_{n\in\Z}X_n^0$ and $\mscr L_k=\bigoplus_{n\in\Z}\mscr L_{k,n}$. Although every fixed block is compact, the full direct sum is compact only if $\norm{\mscr L_{k,n}}_{X_n^0\to X_n^0}\to0$. This necessary condition fails. Indeed, with unit vectors $U_n=(0,e^{\im nx_2})$ in $X^0$,
\[ \mscr L_kU_n=\left(-\im kn\frac{n^2-k^2}{n^2+1}\cos(kx_1)e^{\im nx_2},0\right),\qquad \norm{\mscr L_kU_n}_{X^0}^2=\frac{k^2n^2}{2(1+k^2+n^2)}\left(\frac{n^2-k^2}{n^2+1}\right)^2\longrightarrow\frac{k^2}{2}. \]
The vectors $\mscr L_kU_n$ are mutually orthogonal for distinct $n$, so they have no convergent subsequence and $\mscr L_k$ is not compact. The blockwise compact operators $K_{k,n}$ vary with $n$, and the spectral relation contains the compensating factor $n^2k^2$; accordingly, the full operator has a nontrivial essential spectrum even though each fixed block is compact.

Essential spectrum may be produced by the loss of a uniform resolvent bound as $|n|\to\infty$, even though every fixed block has discrete nonzero spectrum. This high-frequency mechanism gives the imaginary interval in Proposition \ref{prop:sobolev-semigroup} but does not affect the unstable half-plane.
\end{remark}

\subsection{Essential spectrum, Sobolev invariance, and semigroup bounds}

\begin{proposition} \label{prop:sobolev-semigroup}
The spectrum of $\mscr L_k$ is the same on every $X^s$ for $s\geq0$. With the Fredholm definition of essential spectrum,
\[ \sigma_{\mathrm{ess}}(\mscr L_k)=\im[-k\sqrt{1+k^2},k\sqrt{1+k^2}]. \]
If $Y^s\subset X^s$ is a closed invariant subspace defined by Fourier modes or the reflections used below, and if the restricted operator has spectral bound $\Lambda$, then for every $s\geq0$ and $\eta>0$ there is a constant $C_{k,Y,s,\eta}$ such that
\begin{equation} \label{eq:full-semigroup-bound}
\norm{e^{t\mscr L_k}}_{Y^s\to Y^s}\leq C_{k,Y,s,\eta}e^{(\Lambda+\eta)t},\qquad t\geq0.
\end{equation}
\end{proposition}

\begin{proof}
Extend $\Delta^{-1}$ by zero on the constant Fourier mode and let $\Lambda_0=(1-\Delta)^{1/2}$. This defines a classical pseudodifferential realization on all of $L^2\times L^2$ after the identification $(\omega,F)\mapsto(\Lambda_0^{-1}\omega,F)$. Restriction to the invariant mean-zero vorticity space, or any other choice on the constant mode, changes only a finite-dimensional summand and does not alter the Fredholm essential spectrum or the index. The resulting order-zero system $\widetilde{\mscr L}_k$ has principal symbol
\[ -\im kC_k(x_1)\frac{\xi_2}{|\xi|}\begin{pmatrix}0&1\\1+k^2&0\end{pmatrix}, \]
whose eigenvalues range over the interval $I_k=\im[-k\sqrt{1+k^2},k\sqrt{1+k^2}]$ on the cosphere bundle.

If $z\notin I_k$, then $z-\widetilde{\mscr L}_k$ is elliptic. Standard order-zero pseudodifferential Fredholm theory~\cite{Shvydkoy2006} gives $G_z\in\Psi^0$ such that
\[ G_z(z-\widetilde{\mscr L}_k)=I-R_{z,1},\qquad (z-\widetilde{\mscr L}_k)G_z=I-R_{z,2},\qquad R_{z,j}\in\Psi^{-1}. \]
The remainders are compact on $L^2\times L^2$ because they map into $H^1$. By Atkinson's theorem, a two-sided inverse modulo compact operators makes $z-\widetilde{\mscr L}_k$ Fredholm. Its index is zero: the index is locally constant on the connected set $\C\setminus I_k$, and the operator is invertible for large $|z|$.

If $z\in I_k$, choose a characteristic point $(x^0,\xi^0)$ and a null vector $v$ of the principal symbol of $z-\widetilde{\mscr L}_k$. Normalized coherent states of the form $h^{-1/2}\chi((x-x^0)/\sqrt h)e^{\im(x-x^0)\cdot\xi^0/h}v$, including at the endpoints of the symbol range, satisfy
\[ u_h\rightharpoonup0,\qquad \norm{(z-\widetilde{\mscr L}_k)u_h}_{L^2}=O(\sqrt h). \]
A Fredholm operator $T$ satisfies $\norm{u}\leq C\norm{Tu}+\norm{\Pi u}$, where $\Pi$ projects onto its finite-dimensional kernel. Since $\Pi u_h\to0$, this estimate contradicts $\norm{u_h}=1$. Projecting $u_h$ onto mean-zero vorticity changes it by $o(1)$ and gives the same contradiction on $X^0$. Hence $z-\widetilde{\mscr L}_k$ is not Fredholm, proving $\sigma_{\mathrm{ess}}(\mscr L_k)=I_k$.

For $s\geq0$, identify $X^s$ with $L^2\times L^2$ by $(\omega,F)\mapsto(\Lambda_0^{s-1}\omega,\Lambda_0^sF)$. The corresponding realization differs from $\widetilde{\mscr L}_k$ by a commutator term of order $-1$, schematically
\[ [\Lambda_0^s,\widetilde{\mscr L}_k]\Lambda_0^{-s}\in\Psi^{-1}, \]
hence by a compact operator. Thus Fredholmness, index, and the essential spectrum are unchanged. Outside $I_k$, elliptic regularity shows that every kernel element is smooth, so the realizations of $\mscr L_k-\lambda I$ on $X^s$ and $X^{s'}$ have the same kernel for any $0\leq s<s'$. Since their index is zero, they are simultaneously invertible. Points of $I_k$ belong to the essential spectrum on every realization, so the full spectrum is independent of $X^s$.

Finally, let $A=\mscr L_k|_{Y^s}$. The operator $A$ is bounded. If its spectral bound is $\Lambda$, the spectral mapping theorem and the spectral-radius formula give $r(e^A)=e^\Lambda$ and, for each $\eta>0$, a constant $C_\eta$ such that $\norm{e^{mA}}\leq C_\eta e^{(\Lambda+\eta)m}$ for every integer $m\geq0$. Writing $t=m+\tau$ with $0\leq\tau<1$ and using $\sup_{0\leq\tau\leq1}\norm{e^{\tau A}}<\infty$ proves \eqref{eq:full-semigroup-bound}. The Fourier and reflection projections commute with the Sobolev identifications above, so the same regularity argument makes $\Lambda$ independent of $s$ for the subspaces used in this article.
\end{proof}

\subsection{An invariant symmetry class}

Let $R_1(x_1,x_2)=(\frac{\pi}{k} - x_1,x_2)$ and $R_2(x_1,x_2)=(x_1,-x_2)$. Each reflection acts on perturbations by
\[ (\omega,F)(x)\longmapsto(-\omega(Rx),F(Rx)), \]
and similarly on $(\psi,\Phi)$. These actions and the translation by $2\pi/k$ in $x_1$ preserve both the system and the equilibrium. Hence the nonlinear flow preserves their common fixed-point space:
\begin{equation} \label{eq:nonlinear-symmetry-class}
\begin{aligned}
(\omega,F)(x_1+2\pi/k,x_2)&=(\omega,F)(x_1,x_2),\\
\omega(\pi/k-x_1,x_2)&=-\omega(x_1,x_2),& F(\pi/k-x_1,x_2)&=F(x_1,x_2),\\
\omega(x_1,-x_2)&=-\omega(x_1,x_2),& F(x_1,-x_2)&=F(x_1,x_2).
\end{aligned}
\end{equation}
Denote this space by $X^s_{\mathrm{sym}}$; it is invariant under the full nonlinear flow and $\mscr L_k$. Its $x_2$ parity conditions require $\sin(nx_2)$ for $\omega$ and $\cos(nx_2)$ for $F$. For each $n\geq1$, the $x_1$ parity gives two invariant chains. The cosine chain is
\begin{equation} \label{eq:unstable-parity-chain}
F\in\overline{\operatorname{span}}\{\cos(2jkx_1):j\geq0\},\qquad \omega\in\overline{\operatorname{span}}\{\cos((2j+1)kx_1):j\geq0\}.
\end{equation}
The complementary, spectrally non-growing chain is
\begin{equation} \label{eq:stable-parity-chain}
F\in\overline{\operatorname{span}}\{\sin((2j+1)kx_1):j\geq0\},\qquad \omega\in\overline{\operatorname{span}}\{\sin(2jkx_1):j\geq1\}.
\end{equation}
The Fourier multipliers preserve frequency, while multiplication by $C_k$ changes the $x_1$-frequency by $\pm k$; hence both chains are invariant. The sine chain has no spectrum in the open right half-plane, so all later unstable-mode calculations may be restricted to \eqref{eq:unstable-parity-chain}.

\begin{lemma} \label{lem:sine-chain-stability}
For every $n\geq1$, the restriction of $\mscr L_k$ to the subspace in \eqref{eq:stable-parity-chain} has no spectral value with positive real part.
\end{lemma}

\begin{proof}
Write $\omega=h(x_1)\sin(nx_2)$ and $F=f(x_1)\cos(nx_2)$. On the sine chain, \eqref{eq:full-symbols} gives
\[ A_{k,n}\sin((2j+1)kx_1)=\frac{4j(j+1)k^2+n^2}{(2j+1)^2k^2+n^2+1}\sin((2j+1)kx_1),\qquad j\geq0, \]
which is positive. Likewise,
\[ B_{k,n}\sin(2jkx_1)=\frac{1+k^2}{4j^2k^2+n^2}\sin(2jkx_1),\qquad j\geq1, \]
is positive. The real block equations are
\[ \del_t h=nkC_kA_{k,n}f,\qquad \del_t f=-nkC_kB_{k,n}h. \]
Thus the positive quadratic form
\[ E_n(h,f)=\langle B_{k,n}h,h\rangle_{L^2}+\langle A_{k,n}f,f\rangle_{L^2} \]
is conserved:
\[ \frac12\frac{d}{dt}E_n(h,f)=nk\operatorname{Re}\langle B_{k,n}h,C_kA_{k,n}f\rangle-nk\operatorname{Re}\langle A_{k,n}f,C_kB_{k,n}h\rangle=0. \]
For a nonzero $z$-eigenvector, conservation gives $e^{2\operatorname{Re}z t}E_n=E_n$; positivity forces $\operatorname{Re}z=0$. Proposition \ref{prop:one-mode-spectrum} says that every nonzero spectral point of this compact block is an eigenvalue. Hence the sine-chain spectrum lies on the imaginary axis, possibly including zero. The invariant direct-sum decomposition \eqref{eq:unstable-parity-chain}--\eqref{eq:stable-parity-chain} then places every unstable eigenfunction in the cosine chain.
\end{proof}

\begin{proposition} \label{prop:symmetry-branches}
Within the symmetry class \eqref{eq:nonlinear-symmetry-class}, the chain \eqref{eq:unstable-parity-chain} has exactly one positive eigenvalue $\lambda_{k,n}$ when $1\leq n<k$ and none when $n\geq k$. The positive eigenvalue is algebraically simple within its transverse frequency block. If $K_{k,n}^{\mathrm{ev}}$ denotes the restriction of $K_{k,n}$ to the even-cosine $F$ space in \eqref{eq:unstable-parity-chain}, and
\[ K_{k,n}^{\mathrm{ev}}g_{k,n}=-\nu_{k,n}g_{k,n},\qquad f_{k,n}=D_{k,n}^{1/2}g_{k,n}, \]
then a real eigenfunction is
\begin{equation} \label{eq:real-branch-eigenfunction}
F_{k,n}=f_{k,n}(x_1)\cos(nx_2),\qquad \omega_{k,n}=\frac{nk}{\lambda_{k,n}}C_kA_{k,n}f_{k,n}(x_1)\sin(nx_2),\qquad \lambda_{k,n}=nk\sqrt{\nu_{k,n}}.
\end{equation}
\end{proposition}

\begin{proof}
For $1\leq n<k$, the symbols of $A_{k,n}$ on the even-cosine $F$ space are
\[ \frac{n^2-k^2}{n^2+1}<0\quad\text{on }1,\qquad \frac{4j^2k^2+n^2-k^2}{4j^2k^2+n^2+1}>0\quad\text{on }\cos(2jkx_1),\quad j\geq1. \]
Thus $A_{k,n}$ has negative index one. The even-cosine and odd-cosine spaces reduce $D_{k,n}=C_kB_{k,n}C_k$, and therefore also reduce $D_{k,n}^{1/2}$ by functional calculus. The restriction of $D_{k,n}$ to the even-cosine space is injective, so the range of its square root is dense there. Applying both inequalities in the proof of Lemma \ref{lem:negative-index} within this reducing space shows that $K_{k,n}^{\mathrm{ev}}$ has negative index one. Compact self-adjoint spectral theory therefore gives exactly one negative eigenvalue counted with multiplicity, so it is simple. For $n\geq k$, all symbols are nonnegative and no unstable eigenvalue exists. The relation $D_{k,n}A_{k,n}f_{k,n}=-\nu_{k,n}f_{k,n}$ gives \eqref{eq:real-branch-eigenfunction}.

For algebraic simplicity, Proposition \ref{prop:one-mode-spectrum} gives $\lambda_{k,n}^2$ algebraic multiplicity one in each diagonal block of $\mscr L_{k,n}^2$, hence total multiplicity two. Since
\[ \mscr L_{k,n}^2-\lambda_{k,n}^2I=(\mscr L_{k,n}-\lambda_{k,n}I)(\mscr L_{k,n}+\lambda_{k,n}I), \]
the generalized eigenspaces at $\lambda_{k,n}$ and $-\lambda_{k,n}$ form a direct sum of total dimension two. Both are nonzero, so each is one-dimensional.
\end{proof}

\begin{corollary} \label{cor:basic-dichotomy}
If $k=1$, then $\mscr L_1$ has no spectral value with positive real part, and its kernel contains $(0,e^{\pm\im x_2})$. If $k\geq2$, then the $n=1$ branch in $X^s_{\mathrm{sym}}$ gives a positive eigenvalue satisfying
\begin{equation} \label{eq:n-one-lower-bound}
\lambda_{k,1}\geq k\left\{\frac{4k^4-5k^2-3}{4(4k^2+2)}\right\}^{1/2}.
\end{equation}
\end{corollary}

\begin{proof}
For $k=1$, Theorem \ref{thm:full-unstable-spectrum} gives no positive spectral value because there is no nonzero integer $n$ with $|n|<1$. Moreover, $A_1e^{\pm\im x_2}=0$, so $(0,e^{\pm\im x_2})$ belongs to the kernel. Let $k\geq2$ and work on the $n=1$ even-cosine chain. Normalize the $L^2(\mbb T_{x_1})$ inner product by $\norm{1}=1$ and set $g=D_{k,1}^{1/2}1$. Since $B_{k,1}C_k=C_k$,
\[ D_{k,1}1=C_k^2=\frac{1+\cos(2kx_1)}2,\qquad \norm{g}^2=\frac12. \]
The identities
\[ A_{k,1}1=\frac{1-k^2}{2},\qquad A_{k,1}\cos(2kx_1)=\frac{3k^2+1}{4k^2+2}\cos(2kx_1) \]
give
\begin{equation} \label{eq:negative-rayleigh}
\frac{\langle g,K_{k,1}^{\mathrm{ev}} g\rangle}{\norm{g}^2}=\frac{-4k^4+5k^2+3}{4(4k^2+2)}<0.
\end{equation}
The variational principle and $\lambda_{k,1}=k\sqrt{\nu_{k,1}}$ now give \eqref{eq:n-one-lower-bound}.
\end{proof}

\subsection{Continued-fraction characterization and comparison of branches}

We now characterize each unstable branch by a scalar continued fraction and derive quantitative bounds.

Fix $1\leq n<k$. For an eigenfunction in the chain \eqref{eq:unstable-parity-chain}, write
\[ F=f(x_1)\cos(nx_2),\qquad \omega=h(x_1)\sin(nx_2). \]
\[ f=\sum_{m\geq0}f_{2m}\cos(2mkx_1),\qquad h=\sum_{m\geq0}h_{2m+1}\cos((2m+1)kx_1). \]
Put $\mu=2\lambda/(nk)$ and define
\begin{equation} \label{eq:general-c-coefficients}
c_0=\frac{2(n^2-k^2)}{n^2+1},\qquad c_{2m}=\frac{(2m)^2k^2+n^2-k^2}{(2m)^2k^2+n^2+1}\ (m\geq1),\qquad c_{2m+1}=\frac{1+k^2}{(2m+1)^2k^2+n^2}\ (m\geq0).
\end{equation}
If $x_{2m}=f_{2m}$, $x_{2m+1}=h_{2m+1}$, and $q_j=(-1)^{\lfloor j/2\rfloor}x_j$, then multiplication by $C_k$ and the two eigenvalue equations give
\begin{equation} \label{eq:coefficient-recurrence}
\mu q_0=-c_1q_1,\qquad c_0q_0=\mu q_1+c_2q_2,\qquad c_{j-1}q_{j-1}=\mu q_j+c_{j+1}q_{j+1}\quad(j\geq2).
\end{equation}

\begin{proposition} \label{prop:continued-fraction}
For every $\mu>0$ and $j\geq2$, the finite continued fractions defined by $r_{N+1}^{(N)}=0$ and
\[ r_j^{(N)}=\frac{c_{j-1}}{\mu+c_{j+1}r_{j+1}^{(N)}},\qquad 2\leq j\leq N, \]
converge as $N\to\infty$ to a positive continuous function $r_j(\mu)$ satisfying
\begin{equation} \label{eq:general-tail-ratios}
r_j(\mu)=\frac{c_{j-1}}{\mu+c_{j+1}r_{j+1}(\mu)}.
\end{equation}
The recurrence \eqref{eq:coefficient-recurrence} has a nonzero square-summable solution if and only if
\begin{equation} \label{eq:general-matching-condition}
G_{k,n}(\mu):=\mu+c_2r_2(\mu)-\frac{-c_0c_1}{\mu}=0.
\end{equation}
Consequently, $G_{k,n}$ has exactly one positive root $\mu_{k,n}$, and
\begin{equation} \label{eq:branch-eigenvalue}
\lambda_{k,n}=\frac{nk\mu_{k,n}}2.
\end{equation}
Moreover,
\begin{equation} \label{eq:branch-bounds}
\max\{0,-c_0c_1-c_1c_2\}\leq\mu_{k,n}^2<-c_0c_1.
\end{equation}
\end{proposition}

\begin{proof}
We have $0<c_j\leq C$ for $j\geq1$ and $c_{2m+1}\leq C(m+1)^{-2}$, where $C$ depends only on $k$. Each finite tail satisfies $0\leq r_j^{(N)}\leq c_{j-1}/\mu$. If $N'>N$, then $0\leq r_{N+1}^{(N')}\leq c_N/\mu$, while the derivative of $x\mapsto c_{j-1}/(\mu+c_{j+1}x)$ is bounded by $c_{j-1}c_{j+1}/\mu^2$. Propagating the terminal discrepancy backward gives
\[ \abs{r_j^{(N')}-r_j^{(N)}}\leq\frac{C}{\mu}\prod_{\ell=j}^{N}\frac{c_{\ell-1}c_{\ell+1}}{\mu^2}. \]
On every compact subinterval of $(0,\infty)$, the odd-index factors give factorial decay uniformly in $\mu$. This proves convergence, \eqref{eq:general-tail-ratios}, and continuity; positivity of the limit follows from \eqref{eq:general-tail-ratios}.

Set $q_1=1$ and $q_j=r_jq_{j-1}$. Then the tail recurrence holds and $|q_{2m}|+|q_{2m+1}|\leq C_\mu^m/(m!)^2$. For uniqueness, let $p$ and $q$ be square-summable tails and set $W_j=c_{j+1}(q_{j+1}p_j-q_jp_{j+1})$. The recurrence gives $W_j=-(c_{j-1}/c_j)W_{j-1}$ for $j\geq2$ and hence $W_j=(c_{j-2}/c_j)W_{j-2}$ for $j\geq3$. Thus $W_{2m}=(c_2/c_{2m})W_2$ and $W_{2m+1}=(c_1/c_{2m+1})W_1$. Square summability and boundedness of $c_j$ imply $W_j\to0$, whereas $c_{2m}$ stays bounded away from zero and $c_{2m+1}\simeq m^{-2}$; therefore $W_1=W_2=0$, and the tails are proportional.

The first recurrence gives $q_0=-c_1q_1/\mu$; the second then becomes \eqref{eq:general-matching-condition}. A positive root therefore yields a rapidly decreasing eigenfunction. Conversely, the positive eigenvalue from Proposition \ref{prop:symmetry-branches} has a square-summable tail, whose uniqueness forces the ratios $r_j(\mu)$. Hence a root exists, and a second root would contradict uniqueness of the positive eigenvalue. At the root,
\[ \mu_{k,n}^2+\mu_{k,n}c_2r_2(\mu_{k,n})=-c_0c_1. \]
Since $0<r_2(\mu)\leq c_1/\mu$, this identity gives \eqref{eq:branch-bounds}.
\end{proof}

\begin{remark}
The branch $n=1$ need not be the most unstable one. For example, when $k=5$, \eqref{eq:branch-bounds} gives
\[ \lambda_{5,1}^2<150,\qquad \lambda_{5,2}^2\geq 25 \cdot \frac{26}{29}\left(\frac{42}{5}-\frac{79}{105}\right)>171, \]
so $\lambda_{5,2}>\lambda_{5,1}$.
\end{remark}

\subsection{The spectral gap}

Assume $k\geq2$. Within the symmetry class \eqref{eq:nonlinear-symmetry-class}, define
\begin{equation} \label{eq:leading-growth-rate}
\Lambda_k=\max_{1\leq n<k}\lambda_{k,n}.
\end{equation}
Let $\mcal E_k$ be the direct sum of the eigenspaces corresponding to the branches attaining $\Lambda_k$, and define
\[ \Lambda_{k,2}=\max\bigl(\{\lambda_{k,n}:\lambda_{k,n}<\Lambda_k\}\cup\{0\}\bigr),\qquad \gamma_k=\Lambda_k-\Lambda_{k,2}. \]

\begin{theorem} \label{thm:spectral-gap}
On $X^s_{\mathrm{sym}}$ the unstable spectrum is characterized by
\[ \sigma(\mscr L_k)\cap\{\operatorname{Re}z>0\}=\{\lambda_{k,n}:1\leq n<k\}. \]
Each $\lambda_{k,n}$ is algebraically simple within its transverse block. The leading space $\mcal E_k$ is finite-dimensional, $\gamma_k>0$, and the Riesz projection $P_k$ associated with the spectral point $\Lambda_k$ is bounded on every $X^s_{\mathrm{sym}}$. Its range and the action of the linearized operator on it are
\[ \operatorname{Ran}P_k=\mcal E_k,\qquad \mscr L_kP_k=\Lambda_kP_k, \qquad \dim_{\C}\mcal E_k=\#\{n\in\{1,\ldots,k-1\}:\lambda_{k,n}=\Lambda_k\}. \]
More precisely,
\begin{equation} \label{eq:complement-spectrum}
\sigma\bigl(\mscr L_k|_{(I-P_k)X^s_{\mathrm{sym}}}\bigr)\subset\{z\in\C:\operatorname{Re}z\leq\Lambda_{k,2}\}.
\end{equation}
For every $\eta>0$ there is a constant $C_{s,\eta}$ such that
\begin{equation} \label{eq:complement-semigroup-bound}
\norm{e^{t\mscr L_k}(I-P_k)}_{X^s_{\mathrm{sym}}\to X^s_{\mathrm{sym}}}\leq C_{s,\eta}e^{(\Lambda_{k,2}+\eta)t},\qquad t\geq0.
\end{equation}
In particular, there is a constant $C_s$ such that
\begin{equation} \label{eq:leading-semigroup-bound}
\norm{e^{t\mscr L_k}}_{X^s_{\mathrm{sym}}\to X^s_{\mathrm{sym}}}\leq C_se^{\Lambda_kt},\qquad t\geq0.
\end{equation}
The spectral point $\Lambda_k$ is algebraically simple on the full symmetry class if and only if exactly one integer $n_\ast$ attains the maximum in \eqref{eq:leading-growth-rate}. In that case $\mcal E_k$ is one-dimensional in the real symmetry class, and \eqref{eq:real-branch-eigenfunction} with $n=n_\ast$ is a fastest-growing eigenfunction in this symmetry class.
\end{theorem}

\begin{proof}
The $x_2$ parity combines the complex modes $n$ and $-n$ into one real block. Lemma \ref{lem:sine-chain-stability} excludes the sine chain, while Proposition \ref{prop:symmetry-branches} gives one simple eigenvalue $\lambda_{k,n}>0$ in the cosine chain for $1\leq n<k$. This proves the spectral characterization.

There are only $k-1$ positive branches, so $\mcal E_k$ is finite-dimensional and $\Lambda_{k,2}<\Lambda_k$. Theorem \ref{thm:full-unstable-spectrum} shows that every remaining spectral point with positive real part is a smaller positive branch; all other spectral points have nonpositive real part. This proves $\gamma_k>0$ and \eqref{eq:complement-spectrum}.

Choose $r>0$ so small that the closed disk $|z-\Lambda_k|\leq r$ contains no spectral point other than $\Lambda_k$. On $X^s_{\mathrm{sym}}$, define
\[ P_k=\frac{1}{2\pi\im}\int_{|z-\Lambda_k|=r}(z-\mscr L_k)^{-1}\,dz. \]
Proposition \ref{prop:sobolev-semigroup} puts this contour in the resolvent set on every $X^s_{\mathrm{sym}}$, so $P_k$ is bounded there. Complex conjugation commutes with $\mscr L_k$, and the circle is invariant under conjugation, so $P_k\overline U=\overline{P_kU}$ and $P_k$ preserves the real subspace. Transverse blocks are invariant, and each maximizing branch is algebraically simple. Thus $P_k$ selects their direct sum, proving the formulas for its range, dimension, and the criterion for simplicity of $\Lambda_k$.

Spectral mapping on $(I-P_k)X^s_{\mathrm{sym}}$ gives \eqref{eq:complement-semigroup-bound}. Taking $\eta=\gamma_k/2$ yields
\[ \norm{e^{t\mscr L_k}(I-P_k)}\leq C_se^{(\Lambda_k-\gamma_k/2)t}. \]
Since $\mscr L_kP_k=\Lambda_kP_k$ and $P_k$ commutes with $\mscr L_k$,
\[ e^{t\mscr L_k}=e^{\Lambda_kt}P_k+e^{t\mscr L_k}(I-P_k), \]
and hence
\[ \lVert e^{t\mscr L_k} \rVert \leq\norm{P_k}e^{\Lambda_kt}+C_se^{(\Lambda_k-\gamma_k/2)t}\leq C_s'e^{\Lambda_kt}. \]
This proves \eqref{eq:leading-semigroup-bound}.
\end{proof}

\begin{corollary} \label{cor:simple-leading-selection}
Assume that exactly one integer $n_\ast$ attains \eqref{eq:leading-growth-rate}, and fix a real generator $U_\ast$ of $\mcal E_k$. For every $s\geq0$ there is a bounded complex-linear functional $\ell_s$ on $X^s_{\mathrm{sym}}$, taking real values on its real subspace, such that
\[ P_kU=\ell_s(U)U_\ast,\qquad \ell_s(U_\ast)=1,\qquad \ell_s(\mscr L_kU)=\Lambda_k\ell_s(U). \]
Moreover,
\begin{equation} \label{eq:linear-asymptotic-selection}
\norm{e^{-\Lambda_kt}e^{t\mscr L_k}U-\ell_s(U)U_\ast}_{X^s}\leq C_se^{-\gamma_kt/2}\norm{(I-P_k)U}_{X^s},\qquad t\geq0.
\end{equation}
Thus every linear solution in this symmetry class with $\ell_s(U)\neq0$ approaches the same one-dimensional leading eigendirection after multiplication by $e^{-\Lambda_kt}$.
\end{corollary}

\begin{proof}
Simplicity gives $\operatorname{Ran}P_k=\operatorname{span}_{\mathbb C}\{U_\ast\}$. Boundedness and reality of $P_k$ give a bounded functional $\ell_s$, real on real inputs, with $P_kU=\ell_s(U)U_\ast$. The identity $P_k\mscr L_k=\Lambda_kP_k$ gives $\ell_s(\mscr L_kU)=\Lambda_k\ell_s(U)$. Finally,
\[ e^{-\Lambda_kt}e^{t\mscr L_k}U-\ell_s(U)U_\ast=e^{-\Lambda_kt}e^{t\mscr L_k}(I-P_k)U, \]
and the complementary estimate with $\eta=\gamma_k/2$ proves \eqref{eq:linear-asymptotic-selection}.
\end{proof}

\begin{remark}[Choosing a smaller symmetry class]
Imposing $2\pi/n_\ast$-periodicity in $x_2$ preserves the dynamics and the leading eigenfunction while removing transverse harmonics whose indices are not multiples of $n_\ast$. More generally, choosing an integer $n_0$ with $k/2\leq n_0<k$ and imposing $2\pi/n_0$-periodicity leaves exactly one unstable transverse harmonic.
\end{remark}

\section{Nonlinear instability on a logarithmic time scale}

Throughout this section, nonlinear fields and solutions are real-valued. We use the complexified spaces only when discussing the linear spectrum and Riesz projections, which preserve the real subspaces.

Let $U=(\omega,F)$ denote the perturbation of \eqref{eq:pure-equilibrium}, and write
\[ u_U=\grad^\perp\Delta^{-1}\omega,\qquad b_U=\grad^\perp(1-\Delta)^{-1}F. \]
Subtracting the equilibrium equations from \eqref{eq:imh-system} gives
\begin{equation} \label{eq:nonlinear-perturbation-equation}
\del_tU=\mscr L_kU+\mcal B(U,U),
\end{equation}
where the symmetric bilinear map $\mcal B$ is
\begin{equation} \label{eq:symmetric-bilinear-map}
\mcal B(U,V)=-\frac12\begin{pmatrix}u_U\cdot\grad\omega_V+u_V\cdot\grad\omega_U+b_U\cdot\grad F_V+b_V\cdot\grad F_U\\u_U\cdot\grad F_V+u_V\cdot\grad F_U\end{pmatrix}.
\end{equation}
Equivariance under \eqref{eq:nonlinear-symmetry-class} shows that $\mcal B$ preserves this symmetry whenever the expression is defined.

\begin{lemma}[Tame estimates and local well-posedness] \label{lem:nonlinear-energy}
Let $s\geq4$ be an integer. The perturbation equation \eqref{eq:nonlinear-perturbation-equation} is locally well-posed on the real subspace of $X^s_{\mathrm{sym}}$, with solutions in $C([0,T];X^s_{\mathrm{sym}})\cap C^1([0,T];X^{s-1}_{\mathrm{sym}})$. These solutions continue as long as their $X^s$ norm remains bounded. For every integer $r\geq2$,
\begin{equation} \label{eq:quadratic-tame-bound}
\norm{\mcal B(U,V)}_{X^r}\leq C_r\norm{U}_{X^{r+1}}\norm{V}_{X^{r+1}}.
\end{equation}
Moreover, suppose that smooth real-valued functions $V,W,R$ satisfy
\begin{equation} \label{eq:error-equation-general}
\del_tW=\mscr L_kW+2\mcal B(V,W)+\mcal B(W,W)+R.
\end{equation}
Then
\begin{equation} \label{eq:nonlinear-energy-estimate}
\frac{d}{dt}\norm{W}_{X^s}\leq M_s\norm{W}_{X^s}+C_s\bigl(\norm{V}_{X^{s+1}}+\norm{W}_{X^s}\bigr)\norm{W}_{X^s}+\norm{R}_{X^s},
\end{equation}
where $M_s$ and $C_s$ depend only on $s$ and the fixed equilibrium. The same inequality holds in the integrated sense for strong $X^s$ solutions obtained by regularization.
\end{lemma}

\begin{proof}
The elliptic relations give
\[ \norm{u_U}_{H^r}\leq C_r\norm{\omega}_{H^{r-1}},\qquad \norm{b_U}_{H^{r+1}}\leq C_r\norm{F}_{H^r}. \]
Since $H^r(\mbb T^2)$ is an algebra for $r\geq2$,
\[ \norm{u_U\cdot\grad\omega_V}_{H^{r-1}}+\norm{b_U\cdot\grad F_V}_{H^{r-1}}+\norm{u_U\cdot\grad F_V}_{H^r}\leq C_r\norm{U}_{X^{r+1}}\norm{V}_{X^{r+1}}, \]
and the symmetric terms satisfy the same bound. This proves \eqref{eq:quadratic-tame-bound}.

Writing components explicitly gives
\[ \del_t\omega_W=(\mscr L_kW)_\omega-u_V\cdot\grad\omega_W-u_W\cdot\grad\omega_V-u_W\cdot\grad\omega_W-b_V\cdot\grad F_W-b_W\cdot\grad F_V-b_W\cdot\grad F_W+R_\omega, \]
\[ \del_tF_W=(\mscr L_kW)_F-u_V\cdot\grad F_W-u_W\cdot\grad F_V-u_W\cdot\grad F_W+R_F. \]
For real divergence-free $v$, $\operatorname{Re}\langle v\cdot\grad\partial^\alpha f,\partial^\alpha f\rangle=0$, so the self-transport terms contribute only commutators. Use
\[ \norm{[\partial^\alpha,v\cdot\grad]f}_{L^2}\leq C_s\bigl(\norm{\grad v}_{L^\infty}\norm{f}_{H^{s-1}}+\norm{v}_{H^s}\norm{\grad f}_{L^\infty}\bigr),\qquad |\alpha|=s-1, \]
and its $|\alpha|=s$ analogue. The remaining terms contain either a derivative of $V$ or $b_W$, which is one derivative smoother than $F_W$. Sobolev embedding and boundedness of $\mscr L_k$ give
\[ \frac12\frac{d}{dt}\left(\norm{\omega_W}_{H^{s-1}}^2+\norm{F_W}_{H^s}^2\right)\leq M_s\norm{W}_{X^s}^2+C_s\left(\norm{V}_{X^{s+1}}+\norm{W}_{X^s}\right)\norm{W}_{X^s}^2+\norm{R}_{X^s}\norm{W}_{X^s}. \]
Division by $\norm{W}_{X^s}$ proves \eqref{eq:nonlinear-energy-estimate}; at zero, use the upper right derivative. Applying the estimate to Friedrichs regularizations and then using strong convergence in lower norms and lower semicontinuity extends its integrated form to the strong-solution class.

For local existence, solve the Fourier--Galerkin system $\del_tU_N=J_N\mscr L_kU_N+J_N\mcal B(J_NU_N,J_NU_N)$, where $J_N$ is the $L^2$ orthogonal projection onto the Fourier modes in $[-N,N]^2\cap\Z^2$. The same estimate gives
\[ \frac{d}{dt}\norm{U_N}_{X^s}\leq M_s\norm{U_N}_{X^s}+C_s\norm{U_N}_{X^s}^2. \]
Scalar comparison yields an $N$-independent existence time and $X^s$ bound; the equation bounds $\del_tU_N$ in $X^{s-1}$. The compact embedding $X^s\hookrightarrow X^{s-1}$ and the Arzel\`a--Ascoli theorem give, along a subsequence, strong convergence in $C([0,T];X^{s-1})$ and weak-* convergence in $L^\infty([0,T];X^s)$. These convergences are sufficient to pass to the transport products in $C([0,T];X^{s-2})$, recover the initial value, and construct a solution in $L^\infty([0,T];X^s)\cap W^{1,\infty}([0,T];X^{s-1})$.

We prove uniqueness using an estimate one derivative lower than the preceding $X^s$ error estimate. If $U$ and $Z$ are two solutions and $W=U-Z$, the same commutator calculation at level $s-1$ gives
\begin{equation} \label{eq:lower-difference-estimate}
\frac{d}{dt}\norm{W}_{X^{s-1}}\leq C_s\bigl(1+\norm{U}_{X^s}+\norm{Z}_{X^s}\bigr)\norm{W}_{X^{s-1}}.
\end{equation}
Gronwall's inequality gives uniqueness and therefore convergence of the whole Galerkin sequence.

For continuity of the solution map, the standard Bona--Smith argument now applies: solve from smoothed initial data, use \eqref{eq:lower-difference-estimate} for the low-norm difference, and combine the uniform $X^s$ energy bound with persistence estimates for the smoothed solutions to control the high-frequency tails. This yields strong continuity in $X^s$ and continuous dependence on the initial data in $X^s$. Since the right-hand side of \eqref{eq:nonlinear-perturbation-equation} is continuous from $X^s$ to $X^{s-1}$, the solution then belongs to $C^1([0,T];X^{s-1})$. Restarting the uniform energy estimate gives continuation while the $X^s$ norm stays bounded. Since $J_N$ commutes with the defining symmetries and preserves real-valued fields, the limit remains in the real subspace of $X^s_{\mathrm{sym}}$.
\end{proof}

\begin{theorem} \label{thm:nonlinear-instability}
Let $k\geq2$ and let $s\geq4$ be an integer. Choose a real eigenfunction $U_\ast\in\mcal E_k\cap C^\infty$ satisfying
\[ \mscr L_kU_\ast=\Lambda_kU_\ast,\qquad \norm{U_\ast}_{X^s}=1. \]
There are constants $\delta>0$, $\epsilon_0>0$, and $C>0$, depending on $k$, $s$, and $U_\ast$ but not on $\epsilon$, such that for every $0<\epsilon<\epsilon_0$ the nonlinear perturbation equation has a solution $U^\epsilon\in X^s_{\mathrm{sym}}$ with
\[ U^\epsilon(0)=\epsilon U_\ast, \]
on the interval $[0,T_\epsilon]$, where
\begin{equation} \label{eq:escape-time}
T_\epsilon=\frac{1}{\Lambda_k}\log\frac{\delta}{\epsilon}.
\end{equation}
This solution satisfies
\begin{equation} \label{eq:nonlinear-departure}
\sup_{0\leq t\leq T_\epsilon}\norm{U^\epsilon(t)}_{X^s}\leq C\delta,\qquad \norm{U^\epsilon(T_\epsilon)}_{X^s}\geq\frac{\delta}{2}.
\end{equation}
In particular, the purely magnetic equilibrium is nonlinearly unstable in $X^s_{\mathrm{sym}}$, and an initial perturbation of size $\epsilon$ produces a departure of size independent of $\epsilon$ in time $O(|\log\epsilon|)$.
\end{theorem}

\begin{proof}
The gap gives
\[ e^{t\mscr L_k}=e^{\Lambda_kt}P_k+e^{t\mscr L_k}(I-P_k),\qquad \norm{e^{t\mscr L_k}(I-P_k)}\leq C_se^{(\Lambda_k- \frac{\gamma_k}{2})t}, \]
and hence \eqref{eq:leading-semigroup-bound}, even if leading branches tie.

Choose an integer $N\geq2$ so large that
\begin{equation} \label{eq:choice-of-expansion-order}
(N+1)\Lambda_k>M_s + 1.
\end{equation}

To absorb the derivative loss, set $U_1(t)=e^{\Lambda_kt}U_\ast$ and define $U_2,\ldots,U_N$ by
\begin{equation} \label{eq:approximate-profile-recursion}
(\del_t-\mscr L_k)U_j=\sum_{p+q=j}\mcal B(U_p,U_q),\qquad U_j(0)=0.
\end{equation}
Smoothness of $U_\ast$ permits this recursion. With $s_j=s+N-j+1$, induction gives
\begin{equation} \label{eq:profile-growth}
\norm{U_j(t)}_{X^{s_j}}\leq C_j e^{j\Lambda_kt},\qquad 1\leq j\leq N,\quad t\geq0.
\end{equation}
Indeed, for $p+q=j$ one has $s_p,s_q\geq s_j+1$, so \eqref{eq:quadratic-tame-bound} bounds the forcing by $Ce^{j\Lambda_kt}$. Duhamel's formula gives
\[ \norm{U_j(t)}_{X^{s_j}}\leq C\int_0^te^{\Lambda_k(t-\tau)}e^{j\Lambda_k\tau}\,d\tau=C\frac{e^{j\Lambda_kt}-e^{\Lambda_kt}}{(j-1)\Lambda_k}\leq C_je^{j\Lambda_kt}, \]
for $j\geq2$. We now define an approximate unstable trajectory:
\[ U^{\text{app}}(t)=\sum_{j=1}^N\epsilon^jU_j(t),\qquad R^{\text{app}}=(\del_t-\mscr L_k)U^{\text{app}}-\mcal B(U^{\text{app}},U^{\text{app}}),\qquad \theta_\epsilon(t)=\epsilon e^{\Lambda_kt}. \]
The recursion cancels orders $1$ through $N$, leaving
\[ R^{\mathrm{app}}=-\sum_{\substack{1\leq p,q\leq N\\N+1\leq p+q\leq2N}}\epsilon^{p+q}\mcal B(U_p,U_q). \]
For $\theta_\epsilon(t)\leq1$, \eqref{eq:quadratic-tame-bound} and \eqref{eq:profile-growth} yield constants $C_{\mathrm{app}},C_{\mathrm{res}}>0$, independent of $\epsilon$ and $t$, such that
\begin{equation} \label{eq:approximate-solution-bounds}
\norm{U^{\text{app}}(t)}_{X^{s+1}}\leq C_{\mathrm{app}}\epsilon e^{\Lambda_kt},\qquad \norm{R^{\text{app}}(t)}_{X^s}\leq C_{\mathrm{res}}(\epsilon e^{\Lambda_kt})^{N+1}.
\end{equation}

For $W=U^\epsilon-U^{\text{app}}$, one has $W(0)=0$ and
\[ \del_tW=\mscr L_kW+2\mcal B(U^{\text{app}},W)+\mcal B(W,W)-R^{\text{app}}. \]
Set
\[ C_\ast=\frac{C_{\mathrm{res}}}{(N+1)\Lambda_k-M_s - 1}, \]
and choose $0<\delta_0<1$ so small that
\[ C_s(C_{\mathrm{app}}+1)\delta_0\leq1,\qquad C_\ast\delta_0^N\leq\frac12. \]
Let $T_{\max}^\epsilon$ be the maximal existence time of $U^\epsilon$. For $0<\epsilon<\delta_0$, define
\[ T_{\delta_0,\epsilon}=\frac{1}{\Lambda_k}\log\frac{\delta_0}{\epsilon},\qquad \tau_\epsilon=\sup\left\{0<T<\min\{T_{\max}^\epsilon,T_{\delta_0,\epsilon}\}:\sup_{0\leq t\leq T}\norm{W(t)}_{X^s}\leq\delta_0\right\}. \]
For $0\leq t<\tau_\epsilon$, one has $\epsilon e^{\Lambda_kt}\leq\delta_0$ and $\norm{W(t)}_{X^s}\leq\delta_0$. Applying \eqref{eq:nonlinear-energy-estimate} with $V=U^{\text{app}}$ and $R=-R^{\text{app}}$, and then using \eqref{eq:approximate-solution-bounds}, gives
\[ \begin{aligned} \frac{d}{dt}\norm{W}_{X^s}&\leq M_s\norm{W}_{X^s}+C_s\bigl(C_{\mathrm{app}}\epsilon e^{\Lambda_kt}+\norm{W}_{X^s}\bigr)\norm{W}_{X^s}+C_{\mathrm{res}}(\epsilon e^{\Lambda_kt})^{N+1}\\ &\leq (M_s + 1)\norm{W}_{X^s}+C_{\mathrm{res}}\epsilon^{N+1}e^{(N+1)\Lambda_kt}. \end{aligned} \]
Since $W(0)=0$, integration of this differential inequality and \eqref{eq:choice-of-expansion-order} yield, for $0\leq t<\tau_\epsilon$,
\[ \norm{W(t)}_{X^s}\leq C_{\mathrm{res}}\epsilon^{N+1}\int_0^te^{(M_s+1)(t-\sigma)}e^{(N+1)\Lambda_k\sigma}\,d\sigma\leq C_\ast(\epsilon e^{\Lambda_kt})^{N+1}\leq\frac{\delta_0}{2}. \]
This is a strict improvement of the bootstrap assumption. Moreover,
\[ \sup_{0\leq t<\tau_\epsilon}\norm{U^\epsilon(t)}_{X^s}\leq\sup_{0\leq t<\tau_\epsilon}\left(\norm{U^{\text{app}}(t)}_{X^s}+\norm{W(t)}_{X^s}\right)\leq\left(C_{\mathrm{app}}+\frac12\right)\delta_0. \]
Continuity of $W$ rules out termination of the bootstrap by $\norm{W}_{X^s}=\delta_0$, while the last bound and the continuation criterion rule out $T_{\max}^\epsilon\leq T_{\delta_0,\epsilon}$. Consequently $U^\epsilon$ exists on $[0,T_{\delta_0,\epsilon}]$, and
\[ \norm{W(t)}_{X^s}\leq C_\ast(\epsilon e^{\Lambda_kt})^{N+1},\qquad 0\leq t\leq T_{\delta_0,\epsilon}. \]

Choose a fixed $0<\delta\leq\delta_0$ so small that
\[ C\delta^2+C\delta^{N+1}\leq\frac{\delta}{2}, \]
and take $\epsilon_0<\delta$. Then
\[ \norm{U^{\text{app}}(T_\epsilon)-\delta U_\ast}_{X^s}\leq\sum_{j=2}^NC_j\delta^j\leq C\delta^2,\qquad \norm{W(T_\epsilon)}_{X^s}\leq C\delta^{N+1}. \]
Therefore we obtain the lower bound in \eqref{eq:nonlinear-departure}.
\end{proof}

\begin{corollary} \label{cor:robust-simple-instability}
Assume in addition that $\Lambda_k$ is algebraically simple on $X^s_{\mathrm{sym}}$, and normalize the real eigenfunction in Corollary \ref{cor:simple-leading-selection} by $\norm{U_\ast}_{X^s}=1$. Choose an integer $N\geq2$ large enough that $(N+1)\Lambda_k>M_s+1$. For every $R>0$ for which the admissible set below is nonempty, there are constants $\delta,\epsilon_0,C_R>0$ with the following property. If $V\in X^{s+N}_{\mathrm{sym}}$ is real and satisfies
\[ P_k V = U_\ast ,\qquad \norm{V}_{X^{s+N}}\leq R, \]
then for every $0<\epsilon<\epsilon_0$ the solution with $U^\epsilon(0)=\epsilon V$ exists up to $T_\epsilon=\Lambda_k^{-1}\log(\delta/\epsilon)$ and satisfies
\begin{equation} \label{eq:robust-simple-escape}
\norm{U^\epsilon(T_\epsilon)-\delta U_\ast}_{X^s}\leq C_R\left\{\delta\left(\frac{\epsilon}{\delta}\right)^{\gamma_k/(2\Lambda_k)}+\delta^2\right\}.
\end{equation}
After choosing $\delta$ and $\epsilon_0$ sufficiently small, $\norm{U^\epsilon(T_\epsilon)}_{X^s}\geq\delta/2$.
\end{corollary}

\begin{proof}
Use the same $N$, set $U_1(t)=e^{t\mscr L_k}V$, and define the higher profiles by \eqref{eq:approximate-profile-recursion}. Since $P_kV=U_\ast$,
\[ U_1(t)=e^{\Lambda_kt}U_\ast+e^{t\mscr L_k}(I-P_k)V,\qquad \norm{U_1(t)}_{X^{s+N}}\leq C_Re^{\Lambda_kt}. \]
The preceding profile and error estimates are uniform for $\norm{V}_{X^{s+N}}\leq R$:
\[ \norm{U_j(t)}_{X^{s+N-j+1}}\leq C_Re^{j\Lambda_kt},\qquad \norm{W(t)}_{X^s}\leq C_R\bigl(\epsilon e^{\Lambda_kt}\bigr)^{N+1}. \]
After decreasing $\delta$ in terms of $R$, the same bootstrap extends the solution through $T_\epsilon$. At that time,
\[ \norm{\epsilon e^{T_\epsilon\mscr L_k}(I-P_k)V}_{X^s}\leq C_R\epsilon e^{(\Lambda_k-\gamma_k/2)T_\epsilon}=C_R\delta\left(\frac{\epsilon}{\delta}\right)^{\gamma_k/(2\Lambda_k)}. \]
The remaining profiles and the exact error satisfy
\[ \norm{\sum_{j=2}^N\epsilon^jU_j(T_\epsilon)+W(T_\epsilon)}_{X^s}\leq C_R\sum_{j=2}^N\delta^j+C_R\delta^{N+1}\leq C_R\delta^2. \]
These estimates prove \eqref{eq:robust-simple-escape}. Choose first $\delta$ and then $\epsilon_0$ so that the two terms on its right-hand side are at most $\delta/4$; the lower bound follows.
\end{proof}

\begin{remark}[Quadratic cancellation for even perturbations]
Let $X^s_{\mathrm{ev}}$ and $X^s_{\mathrm{od}}$ denote the subspaces of $X^s_{\mathrm{sym}}$ consisting, respectively, of pairs whose components are both even or both odd in $x_1$, and let $\Pi_k^u$ be the Riesz projection onto the unstable subspace. From the definition \eqref{eq:symmetric-bilinear-map},
\[ \mcal B(X^{s+1}_{\mathrm{ev}},X^{s+1}_{\mathrm{ev}})\subset X^s_{\mathrm{od}},\qquad \mcal B(X^{s+1}_{\mathrm{ev}},X^{s+1}_{\mathrm{od}})\subset X^s_{\mathrm{ev}}. \]
By Lemma \ref{lem:sine-chain-stability}, every unstable eigenspace in $X^s_{\mathrm{sym}}$ is even, whereas the odd sector is invariant and spectrally non-growing. Hence $\Pi_k^u\mcal B(U,U)=0$ for every even $U$, in particular for every $U_\ast\in\mcal E_k$. If $G_2=\mcal B(U_\ast,U_\ast)$, the second profile in the expansion is
\[ U_2(t)=e^{2\Lambda_kt}V_2-e^{t\mscr L_k}V_2,\qquad V_2=(2\Lambda_k-\mscr L_k)^{-1}G_2. \]
Both $V_2$ and $U_2(t)$ remain odd, so $\Pi_k^uU_2(t)=0$. Thus the second profile generated by a leading eigenfunction has zero unstable projection and satisfies $\epsilon^2U_2=O((\epsilon e^{\Lambda_kt})^2)$ in $X^s$.
\end{remark}

\subsection*{Acknowledgments}
The author is grateful to Professor Constantin for suggesting this problem and for valuable discussions. He also thanks Zhongtian Hu for a helpful discussion.

\raggedright

\end{document}